\documentclass[12pt]{article}
\usepackage{amssymb,amsmath,amsthm}
\usepackage[usenames]{color}

\newtheorem{thm}{Theorem}[section]
\newtheorem{proposition}[thm]{Proposition}
\newtheorem{lemma}[thm]{Lemma}

\newtheorem{corollary}[thm]{Corollary}
\newtheorem{remark}[thm]{Remark}

\DeclareMathOperator{\Spec}{Spec}
\DeclareMathOperator{\SL}{SL}
\DeclareMathOperator{\PSL}{PSL}

\DeclareMathOperator{\GL}{GL}
\DeclareMathOperator{\Z}{\mathbb{Z}}
\DeclareMathOperator{\Id}{Id}
\DeclareMathOperator{\trace}{trace}

\title{Word values in $p$-adic and adelic groups}

\author{Nir Avni, Tsachik Gelander, Martin Kassabov, Aner Shalev}

\begin{document}

\maketitle

\section{Introduction}

In recent years there has been wide interest in word maps on groups,
with various connections and applications to other areas; see
for instance~\cite{Bo}, \cite{LiSh}, \cite{La}, \cite{Ja}, \cite{Sh1}, 
\cite{Sh2},
\cite{Sh3}, \cite{LSh1}, \cite{LSh2}, \cite{LSh3}, \cite{GSh}, \cite{LBST},
\cite{NP}, \cite{Se2}, \cite{ScSh}, \cite{LST1}, \cite{LST2}, as well as Segal's
book~\cite{Se}
and the references therein.

Recall that a word $w=w(x_1, \ldots , x_d)$ is an element of a free group
$F_d$ on the free generators $x_1, \ldots, x_d$.
Given such a word $w$ and a group $G$ we define
a word map from $G^d$ to $G$ induced by substitution. The image of this
map is denoted by $w(G)$.

The study of this image attracted considerable attention.
By Borel~\cite{Bo} (see also \cite{La}) word maps corresponding to
a non-trivial word $w \ne 1$ are dominant on simple algebraic groups.
Word maps on finite simple groups were subsequently studied extensively.
A major result from~\cite{LST1} shows that, given a word $w \ne 1$, and
a sufficiently large finite simple group $G$, every element of $G$ is a product
of two values of $w$, namely $w(G)^2 = G$.

For some special words even more has been shown. If $w = [x_1,x_2]$, the
commutator word, then~\cite{LBST} shows that $w(G) = G$ for all
finite simple groups, proving a longstanding conjecture of Ore~\cite{O}
(see also~\cite{EG} and the references therein). 

In this paper we study similar problems for certain infinite groups,
namely $p$-adic groups, and, more generally, simple algebraic groups over
local rings. This involves new challenges and new methods.
In~\cite{Sh3} it is conjectured that
every element of $\SL_n(\Z_p)$ ($p$ prime, $n \ge 2$) is a commutator
(assuming $p>3$ if $n = 2$), and some preliminary results were obtained.
In Theorem~\ref{thm:commutatorPSL} below we prove this conjecture for 
$\PSL_n(\mathbb{Z}_p)$ assuming that $n$ is a proper divisor of $p-1$. 
Moreover, considering the very general setup $G=\SL_n(O)$ where $O$ is a local ring whose residue field $O/m$ has more than $n+1$ elements, we show that every matrix which is not congruent to a scalar modulo $m$ can be expressed as a commutator (see Theorem \ref{thm:commutator}).

We also study arbitrary word maps in $p$-adic groups. 
A result of Jaikin-Zapirain \cite{Ja} shows that every word
has finite width in a compact finitely generated $p$-adic analytic
group. This means that every element of the verbal subgroup generated
by the image of the word map is a bounded product of word values
and their inverses. Finding explicit bounds on the word
width is in general a very challenging problem.

Consider the $p$-adic group $G(\Z_p)$, where $G$ is a 
semisimple, simply connected, algebraic group over $\mathbb{Q}$.
For any fixed word $w \ne 1$ and large $p$ we show 
that $w(G(\Z_p))^3 = G(\Z_p)$, so in particular the word width
is at most 3 (see Theorem~\ref{prop:3.words} below,
which is somewhat stronger). It turns out that this result cannot
be improved to $w(G(\Z_p))^2 = G(\Z_p)$, since by~\cite{LST2} finite
quasisimple groups $H$ need not satisfy $w(H)^2 = H$ (when $w \ne 1$ is
fixed and $H$ is as large as we like).
However, we show in Theorem~\ref{prop:2.words}
below that every element of $G(\Z_p)$ whose image modulo the
first congruence subgroup $G^1(\Z_p)$ is not central is a product
of two values of $w$.
We also prove some results on word maps on adelic groups.

\noindent
{\bf Acknowledgement.} We thank Michael Larsen for useful comments. 
Avni was supported by NSF Grant DMS-0901638,
Gelander was supported by ERC grant 203418 and ISF grant 1003/11,
Kassabov was supported by NSF grant DMS-0900932, and Shalev was
supported by advanced ERC grant 247034.

\section{General Words}

In this section, we lift some results about the possible values of word maps 
from algebraic groups over finite fields to algebraic groups over local rings 
of characteristic 0. To simplify notation, assume that $O=\mathbb{Z} _p$. 
If $G$ is an algebraic group over $\mathbb{Z} _p$, we denote the kernel of 
$G(\mathbb{Z} _p) \rightarrow G(\mathbb{F} _p)$ by $G^1(\mathbb{Z} _p)$. 
For any element $w(x_1,\ldots,x_d)=
x_{i_1}^{m_1}x_{i_2}^{m_2}\cdots x_{i_k}^{m_k}\in F_d$, we define a map, 
which we also denote by $w$, from $G^d$ to $G$ by 
$w(g_1,\ldots,g_d)=g_{i_1}^{m_1}g_{i_2}^{m_2}\cdots g_{i_k}^{m_k}$.
Since $G(\mathbb{Z} _p)$ and $G^1(\mathbb{Z} _p)$ are 
$p$-adic analytic and word maps $w$ are analytic, the derivative (or differential) $dw$ is well defined (in terms of local coordinates it is given by the Jacobian matrix).

\begin{lemma} Let $G$ be a semisimple algebraic group 
over $\mathbb{Q}$.
For every non-trivial word $w$, $w(G(\mathbb{Z}_p))$ contains an open subset of $G(\mathbb{Z}_p)$.
\end{lemma}

\begin{proof} A theorem of Borel \cite{Bo} (see also \cite{La}) says that the map 
$w:G^d \rightarrow G$ is dominant. Over an algebraically closed field of 
characteristic 0, like $\overline{\mathbb{Q}_p}$, this is equivalent to 
the existence of a point for which the derivative of $w$ is surjective 
(as a map of $\overline{\mathbb{Q}_p}$ vector spaces). Since the set of 
points in $G^n$ for which $dw$ is surjective is Zariski open and 
$G(\mathbb{Z}_p)$ is Zariski dense, there is a tuple 
$\overrightarrow{g}=(g_1,\ldots,g_d)\in G(\mathbb{Z}_p)^d$ such that 
$d w|_{\overrightarrow{g}}$ is surjective. The lemma now follows from the 
$p$-adic version of the open mapping theorem 
(cf. \cite[Part II, Ch. 3, Sec. 9]{Serre}).
\end{proof}

\begin{remark} This lemma does not hold for local rings of positive characteristic: for example, the image of $\SL_n(\mathbb{F}_q[[t]])$ under the map $x\mapsto x^p$ is contained in the set of all matrices all of whose eigenvalues are $p$-th powers, which is a nowhere dense set.
\end{remark}

\begin{thm} \label{prop:3.words} 
Suppose that $G$ is a 
semisimple, simply connected, algebraic group over $\mathbb{Q}$, 
and that $w_1,w_2,w_3$ are non-trivial words. If $p$ is large enough, then $w_1(G(\mathbb{Z} _p))\cdot w_2(G(\mathbb{Z} _p))\cdot w_3(G(\mathbb{Z} _p))=G(\mathbb{Z} _p)$. In particular  $w(G(\mathbb{Z} _p))^3 = G(\mathbb{Z} _p)$
for $w \ne 1$ and sufficiently large $p$.
\end{thm}

\begin{proof} Let $k=\mathbb{F} _p$. If $p$ is large enough, $G_k$ is simply connected, $\dim G_k=\dim G$, and $G_k(k)$ is isomorphic to the quotient of $G(\mathbb{Z}_p)$ by its first congruence subgroup. Let $n_i$ be the number of letters appearing in $w_i$, $i=1,2,3$. Since $w_1$ is dominant, $\dim w_1 ^{-1} (g) < n_1 \dim(G)$ for every $g\in G_k$. Since the varieties $w_1 ^{-1}(g)$ have bounded degree, we get that there is a constant $c$ such that $|w_1^{-1}(g)|<c|k|^{n_1\dim(G)-1}$ for every $g\in G_k(k)$, and that the set of points $x\in G(k)^{n_1}$ for which the derivative of $w_1$ at $x$ is surjective has size greater than $\frac{1}{c}|k|^{n_1\dim(G)}$.

Fix $g\in G(\mathbb{Z} _p)$. From the previous paragraph, if $p$ is large, there is $x\in G(\mathbb{Z} _p)^{n_1}$ such that $w_1(x)$ is not congruent to any element in $gZ(G)$ modulo $p$, and the derivative of $w_1$ at $x$ is surjective. By Hensel's Lemma, the last condition implies that $w_1(G(\mathbb{Z} _p))$ contains the $G^1(\mathbb{Z} _p)$-coset of $w_1(x)$.
By~\cite[Theorem 3.3]{LSh1}, there are $y\in G(\mathbb{Z} _p)^{n_2}$ and $z\in G(\mathbb{Z} _p)^{n_3}$ such that $w_2(y)w_3(z)$ is congruent modulo $p$ to $w_1(x)^{-1} g$. But then, $w_1(G(\mathbb{Z} _p))\cdot w_2(G(\mathbb{Z} _p))\cdot w_3(G(\mathbb{Z} _p))$ contains the whole $G^1(\mathbb{Z} _p)$-coset of $w_1(x)w_2(y)w_3(z)$. In particular, it contains $g$.
\end{proof}

Note that the condition on $p$ is necessary, since for small $p$
$w_i$ may be an identity on some finite image of $G(\mathbb{Z} _p)$.

The story for product of two word values is more complicated. Indeed, even for finite quasisimple groups, not every element is a product of two word values. However, the exceptional elements must be in the center.

\begin{thm} \label{prop:2.words} Let $w_1,w_2$ be non-trivial words, and let $G$ be a semi-simple simply connected algebraic group over $\mathbb{Q}$. If $p$ is large enough, then $w_1(G(\mathbb{Z} _p))\cdot w_2(G(\mathbb{Z} _p)) \supset G(\mathbb{Z} _p) \setminus (Z(G)\cdot G^1(\mathbb{Z} _p))$.
\end{thm}

To prove the theorem we need some preparations.
The following Lemma is proved in~\cite{EG}:
\begin{lemma} \label{lem:s.s.conj.class} Suppose $G$ is a simply connected Chevalley group over a finite field $k$ of size greater than $4$. Let $C_1,C_2$ be two conjugacy classes of regular elements from the maximally split torus in $G(k)$. Then $C_1\cdot C_2 \supset G(k)\setminus Z(G)$.
\end{lemma}

\begin{lemma} \label{lem:lift} Suppose that $f:G \rightarrow H$ is a homomorphism, $X_1,X_2$ are subsets of $H$ and $X_1\cdot X_2=F\subset H$. Then $f^{-1}(X_1)\cdot f^{-1}(X_2)=f^{-1}(F)$.
\end{lemma}

\begin{proof} Let $g\in f^{-1}(F)$. By assumption, there are $g_1\in f^{-1}(X_1)$ and $g_2\in f^{-1}(X_2)$ such that $f(g)=f(g_1)f(g_2)$. Therefore, $g_1^{-1} g\in g_2\ker(f) \subset f^{-1}(X_2)$. Hence, there is $g_3\in f^{-1}(X_2)$ such that $g=g_1g_3$.
\end{proof}

We will make use of the following:

\begin{lemma} \label{lem:word.value.regular.split} Let $V\subset G^d$ 
be a proper rational sub-variety. For $p$ large enough, there is a $d$-tuple 
$(x_1,\ldots,x_d)\in G^d(k)\setminus V(k)$ such that $w_1(x_1,\ldots,x_d)$ 
is a regular element in a split torus.
\end{lemma}

\begin{proof}
Since $w:G^d\to G$ is dominant and its differential is surjective on a 
Zariski open set, the restriction $w|_{G^d\setminus V}$ is also dominant. 
Let $W\subset G$ be a proper algebraic subset containing 
$G\setminus w(G\setminus V)$.

By the Lang--Weil estimates (see, for example \cite{LW}), $|W(k)| \leq Cp^{\dim G-1}$ for some constant 
$C$. On the other hand, by~\cite{LST1}, there is a constant $c$ such that 
at least $c|G(k)|$ elements that are regular in some split torus are $w_1$-values. Hence, there is a word value that is both a regular element in a 
split torus, and not in $W(k)$.
\end{proof}

\begin{proof}[Proof of Theorem \ref{prop:2.words}] Let $k=\mathbb{F} _p$. If $p$ is large enough then $G_k$ is semisimple, and the map $G(\mathbb{Z} _p) \rightarrow G(k)$ is onto. Moreover, in this case, Borel's theorem implies that the map $w_1:G^n \rightarrow G$ is dominant, and, hence, generically smooth; let $V \subset G^n$ be the locus of points where $dw_1$ is not surjective.

Pick $(g_1,\ldots,g_d)\in G(\mathbb{Z} _p)^d\setminus V(\mathbb{Z} _p)$. The derivative of $w_1$ at $(g_i)$ is a $\mathbb{Z} _p$-linear map between the relative tangent space to $G^d$ at $(g_i)$ over $\Spec \mathbb{Z} _p$, which is identified with $\mathfrak{g}(\mathbb{Z} _p)^d$, and the relative tangent space to $G$ at $w_1(g_i)$, which is identified with $\mathfrak{g}(\mathbb{Z} _p)$. Since the reduction of this map is onto, it follows that the map itself is onto. It follows that the whole coset $w_1(g_i)G^1(\mathbb{Z} _p)$ is contained in the image of $G(\mathbb{Z} _p)^d$ under $w_1$. Since $w_1(G(\mathbb{Z} _p))$ is closed under conjugation, it follows that it contains the pre-image of a regular semi-simple conjugacy class. The same is true for the image of $w_2$. By Lemmas~\ref{lem:s.s.conj.class} and~\ref{lem:lift}, it follows that $w_1(G(\mathbb{Z} _p))\cdot w_2(G(\mathbb{Z} _p))$ contains all elements whose reduction modulo $p$ is not central.
\end{proof}

We now draw some conclusions for adelic groups.
Results on finite word width in some adelic groups were proved
by Segal in \cite{Se2}.

\begin{corollary} Let $G$ be a semi-simple connected algebraic group over $\mathbb{Q} $.

(i) For any two non-trivial words, $w_1$ and $w_2$, the set $w_1(G(\widehat{\mathbb{Z}}))w_2(G(\widehat{\mathbb{Z}}))$ has positive measure in $G(\widehat{\mathbb{Z}})$.

(ii)If $w_3$ is another non-trivial word, then
$w_1(G(\widehat{\mathbb{Z} })) \cdot w_2(G(\widehat{\mathbb{Z} })) \cdot
w_3(G(\widehat{\mathbb{Z}}))$ is open in $G(\widehat{\mathbb{Z}})$.
\end{corollary}

\begin{proof} First we note that, for any $p$ and any non-trivial word $w$, the set $w(G(\mathbb{Z} _p))$ has a non-empty interior. Indeed, since $w$ is dominant, there is $(g_1,\ldots,g_d)\in G(\mathbb{Z} _p)$ such that the derivative of $w$ at $(g_1,\ldots,g_d)$ surjective.

Suppose that Theorem~\ref{prop:2.words} holds for $p>p_0$.
 
The set $K=\prod_{p \leq p_0} w_1(\mathbb{Z}_p)w_2(\mathbb{Z}_p)$ is open, so has positive measure $\lambda_K$ in $\prod_{p \leq p_0}G(\mathbb{Z}_p)$. Since
\[
w_1(\widehat{\mathbb{Z}})w_2(\widehat{\mathbb{Z}})=\prod_p w_1(\mathbb{Z}_p)w_2(\mathbb{Z}_p) \supset K \times \prod_{p>p_0} \left( G(\mathbb{Z}_p)\setminus(Z(G)\cdot G^1(\mathbb{Z}_p)) \right),
\]
we get that the measure of $w_1(\widehat{\mathbb{Z}})w_2(\widehat{\mathbb{Z}})$ is greater than or equal to 
$$
 \lambda_K\cdot\prod_p \left( 1-\frac{|Z(G)|}{|G(\mathbb{F}_p)|}\right) \geq \lambda_K\cdot\prod_p \left( 1-\frac{C} {p^{\dim G}} \right)
$$ 
 for some constant $C$. Since the dimension of $G$ is greater than 1, the product converges to a positive number.

The second claim follows similarly from Theorem~\ref{prop:3.words}.

\end{proof}

\section{The Commutator Word}

In this section we consider the image of the commutator map in special linear groups over local rings of arbitrary characteristic. Let $O$ be a local ring with residue field $K$ and maximal ideal $\mathfrak{m}$. We set $G=\SL_n$, and let $\mathfrak{g},\mathfrak{g}^*$ be the Lie algebra of $G$ and its dual. The Killing form $\langle X,Y \rangle=\trace(XY)$ is conjugation-invariant and non-degenerate for every $K$.

Let $G^k(O)$ denote the $k$-th congruence subgroup of $G(O)$, i.e., the kernel of $G(O) \to G(O/m^k)$. It is well known that $G^k(O)/ G^{k+1}(O)$ is isomorphic (as a $G(O)$-module) to $\mathfrak{g}(K) \otimes_K \mathfrak{m}^k/\mathfrak{m}^{k+1}$, and the action of $G(O)$ on the tensor product is via the action of $G(K)$ on $\mathfrak{g}(K)$.

\begin{proposition}
\label{prop:centralizer}
Let $\bar g_1,\bar g_2 \in G(K)$ be elements such that the group
$H= \langle \bar g_1 ,\bar g_2\rangle \subset G(K)$ does not have any fixed vectors in $\mathfrak{g}^*(K)$. Then for any lift $g$ of $[\bar g_1,\bar g_2]$ to 
$G(O)$ there are
lifts $g_1$ and $g_2$ of $\bar g_1$ and $\bar g_2$ such that $g = [g_1, g_2]$.
\end{proposition}

\begin{proof} It is enough to show that, under the assumptions of the proposition, the derivative of the commutator map at $(\bar g_1,\bar g_2)$ is onto. We first compute the derivative of the commutator map. Suppose that $g_1,g_2\in G$, that $X,Y\in \mathfrak{g}$, and that $\epsilon ^2=0$. Then
\[
[g_1(1+\epsilon X),g_2(1+\epsilon Y)]=g_1(1+\epsilon X)g_2(1+\epsilon Y)(1-\epsilon X)g_1 ^{-1} (1-\epsilon Y)g_2 ^{-1}=
\]
\[
[g_1,g_2]+\epsilon \left( g_1Xg_2g_1 ^{-1} g_2 ^{-1}+ g_1g_2Yg_1 ^{-1} g_2 ^{-1} -g_1g_2X g_1 ^{-1} g_2 ^{-1} -g_1g_2g_1 ^{-1} Y g_2 ^{-1}\right) =
\]
\[
[g_1,g_2] \left( 1+ \epsilon \left( X^{g_2 g_1 ^{-1} g_2^{-1}} + Y^{g_1 ^{-1} g_2 ^{-1}} - X^{g_1 ^{-1} g_2 ^{-1}} - Y^{g_2 ^{-1}}\right) \right) 
\]
So the derivative of the commutator map at $(g_1,g_2)$ is
\[
 \big[ (X,Y) \mapsto \left( X^{g_2}-X \right)^{g_1 ^{-1}} + \left( Y^{g_1 ^{-1}} - Y \right)\big]^{g_2^{-1}}
\]
For any $g\in G$, the image of $Z \mapsto Z^g-Z$ is the orthogonal complement, with respect to the Killing form $\langle\cdot,\cdot\rangle$, of the centralizer of $g$ because
\[
\left( \forall Z \right) \langle Z^g-Z, W \rangle =0 \iff 
\left( \forall Z \right) \langle Z^g, W \rangle= \langle Z,W \rangle \iff 
\left( \forall Z \right) \langle Z, W^g-W \rangle 
\]
\[
\iff W=W^g.
\]
It follows that the orthogonal complement to the image of the derivative of the commutator map is the intersection of $Z(g_1)$ and $Z(g_2)^{g_1 ^{-1}}$, which is the intersection of $Z(g_1)$ and $Z(g_2)$.
\end{proof}

The assumption in \ref{prop:centralizer} is necessary in the foliowing sense:
 
\begin{proposition}
\label{lem:nosolution}
Let $K$ be a finite field and $\bar g \in G(K)$ an element such that, for any $\bar g_1,\bar g_2$ satisfying $\bar g = [\bar g_1, \bar g_2]$,
the group $\langle \bar g_1, \bar g_2\rangle$ has a fixed point in $\mathfrak{g}^*(K)$. Then there exists a local ring $O$ with residue field $K$ and some lift of $\bar g$
to $G(O)$ which is not a commutator in $G(O)$.
\end{proposition}

\begin{proof}
Let $O = K \oplus \bigoplus_{x\in \mathfrak{g}(K)} e_x K$,
where the elements $e_x$ satisfy $\forall x,y\in \mathfrak{g}(K),~e_x e_y =0$. We identify $G(O)$ with $G(K) \times \prod_{x\in \mathfrak{g}(K)}  \mathfrak{g}(K)e_x$, and let $g$ be the element
$g = \bar g  \times \prod_{x\in \mathfrak{g}_K} x e_x\in G(O)$.
If $g$ is a commutator then there exist $\bar g_1$ and $\bar g_2$
such that for any $x$ the equation
$$
x=\left( 1 - \bar g_2 \right) \cdot h_1 + \left( 1 - \bar g_1^{-1}\right) \cdot h_2
$$
has a solutions, but this contradicts the assumption on $\bar g$.

\end{proof}

\begin{corollary}
There exists a local ring $O$ such that not every element in $\SL_n(O)$ is a commutator.
\end{corollary}

\begin{proof} The element $\bar g=I$ satisfies the conditions of Lemma \ref{lem:nosolution}.
\end{proof} 

\bigskip



On the other hand, Proposition \ref{prop:centralizer} implies

\begin{corollary} \label{cor:primitive.scalar} Suppose that $\lambda \in K$ is a primitive $n$-th root of 1, that $\bar g= \lambda I$, and that $g\in \SL_n(O)$ is some lift of $\bar g$. Then $g$ is a commutator in $\SL_n(O)$. 
\end{corollary} 
\begin{proof} By Proposition \ref{prop:centralizer}, it is enough to show that there are $\bar g_1,\bar g_2\in\SL_n(K)$ such that $[\bar g_1,\bar g_2]=\lambda I$ and there is no non-zero vector in $\mathfrak{sl}_n^*(K)$ that is fixed by both $\bar g_1$ and $\bar g_2$. Let $\bar g_1$ be the diagonal matrix with diagonal $1,\lambda, \lambda ^2,\ldots,\lambda ^{n-1}$, and let $\bar g_2$ be the permutation matrix representing the cycle $(1,2,\ldots,n)$. If $X\in \mathfrak{sl}^*_n(K)$ satisfies $\bar g_1 \cdot X= X$, then $X$ is a diagonal matrix. If, in addition, $\bar g_2\cdot X=X$, then $X$ must be scalar. The existence of $\lambda$ implies that $n$ is prime to the characteristic of $K$, so $\trace(X)=0$ implies that $X=0$.
\end{proof} 

\begin{thm} \label{thm:commutator} If $O$ is a local ring whose residue field has more than $n+1$ elements, then every element of $\SL_n(O)$ that is not scalar modulo $\mathfrak{m}$ is a commutator.
\end{thm}

To show this we need the following.

\begin{lemma} \label{lem:conjugate.n=2} If $A\in\GL_2(O)$ is not central modulo $\mathfrak{m}$ and $\alpha \in O$, then there is $X\in \SL_2(O)$ such that $\left( XAX^{-1} \right) _{1,1}=\alpha$.
\end{lemma}

\begin{proof} Suppose $A=\left( \begin{matrix} a&b\\c&d \end{matrix} \right)$. If $X=\left( \begin{matrix} x&y\\z&w \end{matrix} \right)$ has determinant 1, then the top left entry in $XAX^{-1}$ is $axw-bxz+cyw-dyz$. It is easy to see that there is a solution to $axw-bxz+cyw-dyz=\alpha$ and $xw-yz=1$; For instance one can argue as follows. If $b\notin \mathfrak{m}$, taking $y=0, x=w=1$, there is $z$ such that $\det X=1$ and the top left entry of $XAX ^{-1}$ is $\alpha$. If $c \notin \mathfrak{m}$ the same resigning for $A^t$ applies. If both $b$ and $c$ are in $\mathfrak{m}$, in view of  Hensel's lemma, it is enough to show that there is a rational point of the intersections of the reductions of the two equations modulo $\mathfrak{m}$ where the equations intersect transversely. Choose $\beta$ and $\gamma$ such that $a \beta -d \gamma=\alpha$ and $\beta - \gamma=1$, and choose $x,y,z,w$ such that $xw=\beta$ and $yz=\gamma$.
\end{proof}

In the following, if $A$ is an $n$-by-$n$ matrix, we write $A=\left( \begin{matrix} a&u^T\\v&B \end{matrix} \right)$, where $a$ is a scalar, $B$ is an $(n-1)$-by-$(n-1)$ matrix, and $u$ and $v$ are column vectors.

\begin{lemma} \label{lem:conjugate.big.n} Suppose that $n \geq 3$, that $A\in\GL_n(O)$ is not scalar modulo $\mathfrak{m}$, and that $a \in O^ *$. Then there is a conjugate of $A$ of the form $\left( \begin{matrix} a & w^T\\ z & C \end{matrix} \right)$, where the matrix $a C-zw^T$ is non-scalar modulo $\mathfrak{m}$.
\end{lemma}

\begin{proof} Suppose that $A=\left( \begin{matrix} \alpha & u^T\\ v & B \end{matrix} \right)$. Since $A$ is non-scalar modulo $\mathfrak{m}$, we can conjugate it so that $v$ is non-zero modulo $\mathfrak{m}$. For every $x\in O^{n-1}$,
\[
\left( \begin{matrix} 1&x^T\\0&I \end{matrix} \right) \left( \begin{matrix} \alpha & u^T\\ v & B \end{matrix} \right) \left( \begin{matrix} 1&-x^T\\0&I \end{matrix} \right) = \left( \begin{matrix} \alpha+x^Tv & u^T+x^TB-\alpha x^T- (x^Tv)x^T \\ v&B-vx^T \end{matrix} \right),
\]
and therefore we can find $x$ such that $\alpha+x^Tv=a$. Conjugating by this matrix, we can assume that $A=\left( \begin{matrix} \alpha & u^T\\ v & B \end{matrix} \right)$, where $\alpha=a$ (and $u,v,B$ might have changed). Now it is enough to find $x$ such that \begin{enumerate}
\item $x^Tv=0$.
\item The matrix $\alpha B - vu^T-vx^TB$ is non-scalar modulo $\mathfrak{m}$.
\end{enumerate}
Denote the matrix $\alpha B - v u^T$ by $M$. If the first condition holds, then the expresion in the second condition is equal to $M - \alpha^{-1}  vx^T M$.

If $M$ is non-scalar modulo $\mathfrak{m}$, we can take $x=0$. Otherwise, we can take any $x$ such that $x^Tv=0$ since the sum of a scalar matrix and a non-zero rank-one matrix is never scalar.
\end{proof}

\begin{proof}[Proof of Theorem \ref{thm:commutator}] Suppose that $A\in\GL_n(O)$ is not central modulo $\mathfrak{m}$. Choose $a_1,\ldots,a_n\in O$ such that $\prod a_i=\det(A)$. We shall prove, by induction on $n$, that there is a lower-triangular unipotent matrix $X$, an upper-triangular unipotent matrix $Y$, and a matrix $g\in\SL_n(O)$ such that $XA^gY=D$, where $D$ is the diagonal matrix with diagonal $a_1,\ldots,a_n$. We start with $n=2$. By Lemma \ref{lem:conjugate.n=2}, there is $g\in\SL_2(O)$ such that $(A^g)_{1,1}=a_1$. There is a unipotent lower-triangular matrix $X$ such that $XA^g$ is upper-triangular. Similarly, there is a unipotent upper triangular matrix $Y$ such that $XA^gY$ is diagonal with diagonal $a_1,a_2=\det(A)/a_1$. For the induction step, Lemma \ref{lem:conjugate.big.n} implies that there is $g$ such that $A^g$ has the form $\left( \begin{matrix} a_1 & u^t \\ v &B \end{matrix} \right)$. Let $X=\left( \begin{matrix} 1 & 0 \\ -v/a_1 & \Id \end{matrix} \right)$ and $Y= \left( \begin{matrix} 1 & -u/a_1 \\ 0 & \Id \end{matrix} \right)$. Then $X A^g Y= \left( \begin{matrix} 1 & 0 \\ 0 & C \end{matrix} \right)$, where $C$ is non-scalar modulo $\mathfrak{m}$. Applying the induction hypothesis to $C$, we get the claim for $A$.

Now assume that $A\in\SL_n(O)$ is not a scalar modulo $\mathfrak{m}$, and choose the $a_i$ to satisfy that $\prod a_i=1$, that $a_i=a_{n+1-i}^{-1}$ for all $i=1,\ldots,n$ and that $a_i$ is not congruent to $a_j$ modulo $\mathfrak{m}$ if $i \neq j$ (this is where we use the assumption that the residue field has size greater than $n+1$). By what we have shown, there is a lower-triangular unipotent matrix $X$, an upper-triangular unipotent matrix $Y$, and a matrix $g\in\SL_n(O)$ such that $X^{-1}A^gY^{-1}=D^2$, where $D$ is the diagonal matrix with diagonal $a_1,\ldots,a_n$. Since $A^g=(XD)(DY)$ and since $XD$ and $(DY)^{-1}$ are conjugate (both are cyclic and have the same eigenvalues), it follows that $A^g$ is a commutator. Hence $A$ is also a commutator.
\end{proof}

\begin{thm} \label{thm:commutatorPSL}
If $O$ is a local ring whose residue field $K$ elements, where $n$
is a proper divisor of $|K| -1$, then every element in $\PSL_n(O)$ 
is a commutator.
\end{thm}

\begin{proof} By assumption, there is a primitive $n$-th root of 1 in $K$; denote it by $\lambda$. Let $g\in\PSL_n(O)$. If $\bar g$ is not the identity, let $h\in\SL_n(O)$ be any lift of $g$. If $\bar g=1$, let $h\in\SL_n(O)$ be a lift of $g$ such that $\bar h=\lambda I$. By either Proposition \ref{thm:commutator} or Corollary \ref{cor:primitive.scalar}, the element $h$ is a commutator, and, hence, so is $g$.
\end{proof}

\end{document}